\documentclass[12pt,reqno]{amsart}
\usepackage[top=1in, bottom=1in, left=1in, right=1in]{geometry}

\usepackage{times, amsthm, amssymb, amsmath, amsfonts, graphicx, mathrsfs}
\usepackage{mathtools}
\usepackage{tikz} 
\usetikzlibrary{arrows, cd, matrix, positioning}
\usepackage{xcolor}
\usepackage{bbm}
\usepackage{xypic}
\usepackage{mathscinet}
\usepackage{hyperref}
\usepackage{wrapfig}
\usepackage{subcaption}

\newtheorem{theorem}{Theorem}[section]

\newtheorem{proposition}[theorem]{Proposition}
\newtheorem{corollary}[theorem]{Corollary}

\newtheorem{introtheorem}{Theorem}

\theoremstyle{definition}
\newtheorem{definition}[theorem]{Definition}
\newtheorem{example}[theorem]{Example}
 
\theoremstyle{remark}
\newtheorem{remark}[theorem]{Remark}

\newtheorem{philosophy}[theorem]{Philosophy}

\DeclareMathOperator{\Sym}{Sym}

\DeclareMathOperator{\reg}{reg}

\newcommand{\rank}{\operatorname{rank}}

\newcommand{\Tor}{\operatorname{Tor}}

\newcommand{\gr}{\operatorname{gr}}
\newcommand{\grm}{\operatorname{gr_\frakm}}
\newcommand{\initial}{\operatorname{in}}
\newcommand{\Poin}{\operatorname{Poin}}

\newcommand{\grmR}{{\grm(R)}}

\newcommand{\bbN}{\mathbbm{N}}

\newcommand{\bbP}{\mathbbm{P}}

\newcommand{\bbZ}{\mathbbm{Z}}

\newcommand{\bbk}{\mathbbm{k}}

\newcommand{\calO}{{\mathcal{O}}}

\newcommand{\calS}{{\mathcal{S}}}

\newcommand{\frakm}{{\mathfrak{m}}}

\newcommand{\bfF}{{\mathbf{F}}}
\newcommand{\bfG}{{\mathbf{G}}}

\newcommand{\bfw}{{\mathbf{w}}}
\newcommand{\bfx}{{\mathbf{x}}}
\newcommand{\bfy}{{\mathbf{y}}}

\newcommand{\ideal}[1]{{\langle #1 \rangle}}
\newcommand{\set}[1]{{\{ #1 \}}}

\title{Rational Normal Curves in Weighted Projective Space}

\subjclass{13D02}

\author{Caitlin M. Davis}
\address{University of Wisconsin-Madison, 480 Lincoln Drive, Madison, WI 53706}
\email{cmdavis22@wisc.edu}

\author{Aleksandra Sobieska}
\address{Marshall University, 
One John Marshall Drive, 
Huntington, WV 25755}
\email{sobieskasnyd@marshall.edu}

\begin{document}

\maketitle

\begin{abstract}
This article aims to extend classical homological results about the rational normal curves to analogues in weighted projective spaces.  Results include determinantality and nonstandard versions of quadratic generation and the Koszul property.
\end{abstract}

\section{Introduction}

This article aims to extend classical homological results about the rational normal curves to their weighted analogues.  The rational normal curves provide a rich source of examples of many homological properties. In particular, the defining ideals of these curves are determinantal (\cite[Proposition~6.1]{eGeomOSyz}) with a quadratic Gr\"obner basis (\cite{sGrobnerBasesConvexPolytopes, eInitialIdealsVeroneseSubrings}, \cite[Theorem~3.1.1]{cdrKoszulAlgRegularity}), and their coordinate rings are Cohen-Macaulay and Koszul.  We consider certain analogues of the rational normal curves in weighted projective spaces, and we show that (weighted variants of) all of these properties hold for these curves.

There has been great interest in recent years in extending homological properties to the nonstandard graded and multigraded settings, where we use nonstandard graded to refer to rings with a nonstandard $\bbZ$-grading. Since the generalization by Maclagan-Smith of Castelnuovo-Mumford regularity to multigraded polynomial rings (\cite{msMultigradedCMReg}), there have been many results on multigraded regularity: \cite{bcCMRegMultigradedIdeals,bhsCharacterizingMultigradedReg,bhsBoundsMultigradedReg,chmMultigradedTor,cnMultigradedRegCI,svwMultigradedRegCoarsenings}. There has also been much work on multigraded syzygies: \cite{beLinSyzOCurvesIWProjSp,bePositivityNonstdBetti,beTateRes,besVirtualRes,bklyHomologicalCombinatorial,bsShortResDiagonal,bcnMultigradedSylvesterForms,cSyzygiesCurvesProductsProjSpaces,eesTateResProductsProjSpace,hnvVirtualResPoints,hsMultigradedHilbertSchemes,hssSyzygiesMultigradedReg,svMultigradedRegSyzygiesFatPoints,yVirtualResMonomialIdeals}. In particular, the generalization of the rational normal curve studied in this article comes from Brown and Erman in \cite{beLinSyzOCurvesIWProjSp}. 

Despite this interest in the multigraded setting, few geometric examples have been examined thoroughly.  Existing work concerning geometric examples focuses primarily on sets of points: \cite{hnvVirtualResPoints,svMultigradedRegSyzygiesFatPoints}.  Part of what makes the rational normal curves so important in the literature is that they can be understood from so many different perspectives. They provide an important class of examples for understanding geometric, algebraic, and homological topics, and they are useful
for testing out new definitions and conjectures. Our aim is to develop an analogous class of examples that could play a similar role for the nonstandard graded setting as that theory continues to expand.  With this in mind, we will not focus on the most general possible class of analogues, but rather work to build a family that can be as fully understood as the classical rational normal curves. To be specific:  the Brown-Erman construction embeds a smooth curve based on additional data $(L,D,e)$ where $L$ is a line bundle, $D$ is a base locus divisor, and $e$ is a positive integer. In this article, we study closely the two-parameter family of rational curves where $L=\calO(d)$ varies over all line bundles and $e$ varies over all positive integers, but where $D$ is always simply the point at infinity. The following example illustrates some of the properties which these curves share with the usual rational normal curves.

\begin{example}
\label{ex:d3e2}
    When $d=3$ and $e=2$, the weighted rational curve (see Definition \ref{def:WeightedRationalCurve}) is obtained by the map $\psi: \mathbb P^1 \hookrightarrow \mathbb P(1,1,1,2,2)$ given by $[s:t] \mapsto [s^3:s^2t:st^2:st^5:t^6]$.  Note that the coordinate ring $R \cong \bbk[s^3,s^2t,st^2,st^5,t^6]$ is not normal; for example, $\frac{st^5}{st^2} = t^3$ is integral, but is not in $R$.
    
    However, this curve shares many other properties with the usual rational normal curve.  For instance: $R$ is Cohen-Macaulay; the defining ideal $I$ is generated by the $2\times 2$ minors of
        \[
        \begin{bmatrix} x_0 & x_1 & x_2^2 & y_0 \\ x_1 & x_2 & y_0 & y_1 \end{bmatrix};
        \]
    these generators are a Gr\"obner basis for $I$ with respect to some monomial order (see Section \ref{sec:GB}); and $R$ is resolved by an Eagon-Northcott complex.  
\end{example}

In fact, the observations made in the example above hold for all weighted rational curves.

\begin{introtheorem}
\label{thm:intro}
    Let $R = S/I$ be the coordinate ring of the weighted rational curve of type $d,e$.  Then we have the following:
    \begin{enumerate}
        \item $R$ is Cohen-Macaulay (Proposition \ref{prop:CohenMacaulay})
        \item $I$ is determinantal in the following sense:  $I$ is generated by the $2\times2$ minors of a $2 \times (d+e)$ matrix, almost all of whose entries are linear  (See Theorem \ref{thm:determinantal} for a more precise version.),
        \item $R$ is resolved by the Eagon-Northcott complex (Corollary \ref{cor:eagonnorthcott}), and
        \item $I$ has an initial ideal of quadratic monomials (Theorem \ref{thm:GBforI}). 
    \end{enumerate}
\end{introtheorem}

Note that (2) and (4) are the key points.  Statement (3) follows from (2) and the observation that $I$ has the expected codimension for an ideal of minors.  Statement (1) was first proven in \cite{beLinSyzOCurvesIWProjSp} for curves of arbitrary genus, but our methods provide an alternate proof in genus 0 which more closely parallels the algebraic proofs for rational normal curves. Ongoing work of Banks--Ramkumar adds to these parallels by showing that these weighted rational curves are of minimal degree, just as in the classical setting (cf. \cite[p. 341]{eSyzygiesDegreesChoices}).

Theorem~\ref{thm:intro} shows how weighted rational curves broadly behave as rational normal curves do, yet it conceals the fact that the internal details are quite different.

\begin{example}\label{ex:d6e2}
    Consider the weighted rational curve with $d=6$ and $e=2$ in light of Theorem~\ref{thm:intro}.  This curve is given by $\psi: \mathbb P^1 \hookrightarrow \mathbb P(1,1,1,1,1,1,2,2)$ where
    \[
    [s:t] \mapsto [s^6:s^5t:\cdots:st^5:st^{11}:t^{12}].
    \]
    Theorem \ref{thm:intro}(3)  allows us to determine the Betti table of $S/I$ over $S$:
    \[\begin{matrix}
 & 0 & 1 & 2 & 3 & 4 & 5 & 6\\
%\text{total:} & 1 & 21 & 70 & 105 & 84 & 35 & 6\\
0: & 1 & . & . & . & . & . & .\\
1: & . & 10 & 20 & 15 & 4 & . & .\\
2: & . & 10 & 40 & 60 & 40 & 10 & .\\
3: & . & 1 & 10 & 30 & 40 & 25 & 6
\end{matrix}\]
In the above we follow the formatting used in \texttt{Macaulay2}, so the entry in column $i$ and row $j$ is $\beta_{i,i+j}(R)$ and a dot indicates an entry is zero. This Betti table exhibits ``cancellation'' not found in those of the usual rational normal curves.  That is, syzygies of the same degree show up in multiple homological degrees.  Said another way, one could not predict this Betti table entirely from the Hilbert function, which is a contrast with what happens for rational normal curves.
\end{example}

Another important property of the usual rational normal curve is Koszulness.  The Koszul property, normally defined as $\reg_R(\bbk) = 0$, does not hold for any nonstandard graded ring, and so does not hold for the coordinate ring $R$ of a weighted rational curve. In fact, the differentials in the resolution of $k$ over $R$ contain entries which are not linear combinations of the variables.  In this sense, $R$ is even further from being classically Koszul than, say, a nonstandard graded polynomial ring.
 However, a closer examination of the resolution of $\bbk$ over $R$ reveals that $R$ is ``not too far'' from Koszul, and thus might have an analogous property.

\begin{example}
    When $d=1$ and $e=2$, the weighted rational curve is obtained by the map $\psi: \mathbb P^1 \hookrightarrow \mathbb P(1,2,2)$ given by $[s:t] \mapsto [s:st:t^2]$. The first few columns of the Betti table of $k$ over the coordinate ring $R = \bbk[s,st,t^2] = \bbk[x_0, y_0, y_1]/\ideal{y_0^2 - x_0^2y_1}$ are:
\[
\begin{matrix}
 & 0 & 1 & 2 & 3 & 4\\
%\text{total:} & 1 & 3 & 4 & 4 & 4\\
0: & \mathbf{1} & 1 & . & . & .\\
1: & . & \mathbf{2} & 2 & . & .\\
2: & . & . & \mathbf{2} & 2 & .\\
3: & . & . & . & \mathbf{2} & 2\\
4: & . & . & . & . & \mathbf{2}
\end{matrix}
\]
In Corollary \ref{cor:bettitable}, we describe some of the Betti numbers of $\bbk$ over $R$.  In particular, we describe the linear strand and the ``maximally twisted'' strand (in bold in above example).  We can see immediately that $\reg_R(\bbk)$ is nonzero, so $R$ is not Koszul.  Nonetheless, the resolution of $\bbk$ over $R$ is in some sense ``not too far'' from linear: 
\[
R \xleftarrow{\begin{bmatrix} x & y_0 & y_1 \end{bmatrix}} R(-1)\oplus R(-3)^2 \xleftarrow{\begin{bmatrix} 0 & 0 & -y_0 & -y_1\\ -y_0 & -y_1 & x & 0\\ x^2 & y_0 & 0 & x \end{bmatrix}} R(-3)^2\oplus R(-4)^2 \gets \cdots 
\]
In Section \ref{sec:Koszulness}, we show that $R$ satisfies a generalized Koszul property, introduced by Herzog, Reiner, and Welker in \cite[Definition~5.1]{hrwTKoszulPropIAffSemigroupRings}.
\end{example}

\begin{introtheorem}
\label{thm:IntroKoszul}
    The coordinate ring $R$ of a weighted rational curve is generalized Koszul, in the sense of \cite{hrwTKoszulPropIAffSemigroupRings}.  That is, $\grmR$ is Koszul in the usual sense (Theorem \ref{thm:Koszul}).
\end{introtheorem}

Theorem \ref{thm:IntroKoszul} and Theorem \ref{thm:intro}(4) yield a positive answer to Question~7.9 in \cite{beLinSyzOCurvesIWProjSp} in the case of weighted rational curves. 

Taken together, Theorems \ref{thm:intro} and \ref{thm:IntroKoszul} echo the following theme, which is often seen in the nonstandard graded and multigraded settings.

\begin{philosophy}\label{phil:multigradedSyz}
In the study of nonstandard graded and multigraded syzygies, many classical results have multigraded analogues.   However, finding these analogues can often require subtle deformations of the classical statements, techniques, or definitions. Key examples of this phenomenon can be found in \cite{besVirtualRes}, \cite{msMultigradedCMReg}, and \cite{beTateRes}.
\end{philosophy}

More specifically, we see how things change: in Theorem \ref{thm:intro}(2), the linear matrix is replaced by one which is ``almost'' linear; in Theorem \ref{thm:intro}(4), the Gr\"obner basis is no longer quadratic in the usual sense, but the initial terms remain quadratic; in Theorem \ref{thm:IntroKoszul}, classical Koszulness of $R$ is replaced by Koszulness of the associated graded ring.  Since we will allow the case of weighted rational curves of type $d,1$ (that is, usual rational normal curves), it is also worth noting both that these changes are meaningful only when $e\neq 1$ and that the proofs included here do apply to the usual rational normal curves.

This paper is organized as follows: In \S \ref{sec:Background}, we define the main object of study, the weighted rational curves.  In \S \ref{sec:determinantal}, we prove that the defining ideal of each of these curves is determinantal.  In \S \ref{sec:GB}, we establish that the generators of these ideals are in fact Gr\"obner bases with respect to a certain monomial order.  In \S \ref{sec:Koszulness}, we consider a nonstandard version of the Koszul property and show that it holds for the weighted rational curves.

\subsection*{Acknowledgments} The authors would like to thank Daniel Erman and Maya Banks for their helpful conversation and valuable insights. Most, if not all, computations were performed with the aid of Macaulay2 \cite{M2}.

\section{Background}
\label{sec:Background}

As a way to understand nice embeddings into weighted projective space, Brown and Erman \cite{beLinSyzOCurvesIWProjSp} describe weighted analogues of linear series and of complete linear series.  Their analogue of a complete linear series is a ``log complete series'' and it depends on three pieces of data: a line bundle $L$, an effective divisor $D$, and a positive integer $e$.  We refer the reader to Section~3 of their paper for details.  As stated in the introduction, we will focus our attention on one specific case of their main definition in order to develop a deep understanding of a concrete class of examples in the nonstandard graded setting.  Here we describe this special case: the case where $D$ is the point at infinity, which will be the only case of interest throughout this paper.

We first recall a compact notation for certain weighted projective spaces.  We use exponents to indicate the number of weights of a particular degree.  For example, we write $\mathbb P(1^3,2^2)$ to denote the weighted projective space $\mathbb P(1,1,1,2,2)$ from Example \ref{ex:d3e2}.

\begin{definition}
\label{def:WeightedRationalCurve}
Let $Z=\bbP^1$, $D = [0:1]$, and $L = \mathcal O_{\mathbb P^1}(d)$ for some $d \ge 1$. The corresponding log complete series for a fixed $e \geq 1$ is
\[
W = \ideal{s^d,s^{d-1}t,\ldots,st^{d-1},s^{e-1}t^{de-(e-1)},\ldots st^{de-1},t^{de}}.
\]
This induces a closed immersion $\psi_W \colon \mathbb P^1 \hookrightarrow \mathbb P(1^d,e^e)$ given by 
\[\psi_W \colon [s:t] \mapsto [s^d:s^{d-1}t:\ldots:st^{d-1}:s^{e-1}t^{de-(e-1)}:\ldots:st^{de-1}:t^{de}].\] 

For $W$ and $\psi_W$ as above, we call the image $\psi_W(\bbP^1)$ the \textit{weighted rational curve} of type $(d,e)$.
\end{definition}

Note that in the case that $e=1$, this recovers the usual rational normal curve of degree $d$.

Define $S = \bbk[x_0,\ldots,x_{d-1},y_0,\ldots,y_{e-1}]$ with the weighted grading where $\deg(x_i) = 1$ and $\deg(y_i) = e$, so that it is the coordinate ring of $\bbP(1^d,e^e)$. The immersion $\psi_W$ induces a corresponding ring map $\varphi_W: S \rightarrow \bbk[s,t]$ given by $x_i \mapsto s^{d-i}t^i$ and $y_i \mapsto s^{e-1-i}t^{ed-(e-1)+i}$. 

We will refer to the coordinate ring of the weighted rational curve as $R$ and its defining ideal as $I = \ker \varphi_W$, so that $R = S/I$. While $R$ and $I$ typically depend on $W$, the choice of $W$ will be fixed throughout the paper, and so we omit this dependence from the notation.

\iffalse
\begin{example}
Let $d=3$ and $e=2$. In this case, $W = \langle s^3, s^2t, st^2, st^5, t^6 \rangle$ and $\psi: \mathbb P^1 \hookrightarrow \mathbb P(1,1,1,2,2)$. The corresponding ring map $\varphi: S \rightarrow \bbk[s,t]$ is 
\[
\begin{aligned}
\varphi \colon \bbk[x_0,x_1,x_2, y_0,y_1] &\rightarrow \bbk[s,t] \\
%(x_0,x_1,x_2,y_0,y_1) &\mapsto (s^3,s^2t,st^2,st^5,t^6) \\
x_i &\mapsto s^{3-i} t^i \\
y_i &\mapsto s^{1-i} t^{5+i}.
\end{aligned} 
\]
The corresponding weighted rational curve has coordinate ring $R=S/I$, where $I = \ker\varphi = \ideal{x_0x_2-x_1^2, x_0y_0 - x_1x_2^2, x_0y_1 - x_1y_0, x_1y_0 - x_2^3, x_1y_1 - x_2y_0, x_2^2y_1 - y_0^2}$. Note that $I$ is homogeneous with respect to the nonstandard grading of $S$ where $\deg(x_i)=1$ and $\deg(y_i)=2$. 
\end{example}
\fi

\begin{example}
Let $d=3$ and $e=2$, as in Example \ref{ex:d3e2}. In this case, $W = \langle s^3, s^2t, st^2, st^5, t^6 \rangle$.
\end{example}

The following observations follow easily from \cite{beLinSyzOCurvesIWProjSp}. 

\begin{proposition}
\label{prop:notnormal}
\cite[Remark~3.15]{beLinSyzOCurvesIWProjSp}  The coordinate ring $R$ of a weighted rational curve of type $(d,e)$ is not integrally closed for $e\geq 2$.
\end{proposition}

\begin{proof}
The monomial $t^d$ is in the field of fractions of $R$ since it can be written as 
\[
t^d = \frac{s^{e-1}t^{de - (e-1)}}{(st^{d-1})^{e-1}}.
\]
Furthermore, $t^d$ is integral over $R$.  In particular, it is a root of the monic polynomial $X^e-t^{de}$.  However, $t^d$ is not in $R$ since $e \geq 2$, so $R$ is not integrally closed.
\end{proof}

The above proposition implies that the weighted rational curve is not projectively normal in the usual sense. However, Brown and Erman also show that $R$ is ``normally generated'' in the sense of their Definition 3.13, which parallels Mumford's definition (\cite{mVarietiesQuadraticEquations}), which algebraically implies the following.  
\begin{proposition}
\label{prop:CohenMacaulay}
\cite{beLinSyzOCurvesIWProjSp} $R$ is Cohen-Macaulay.
\end{proposition}

This is a combination of Remark~3.14(1) and Theorem~1.4 in \cite{beLinSyzOCurvesIWProjSp}.  We will offer an alternative proof of Cohen--Macaulayness in Corollary \ref{cor:alternateCohenMacaulay}.

\section{Determinantality}
\label{sec:determinantal}

In this section, we show that the weighted rational curve is defined by the maximal minors of a certain matrix, as in the case of the usual rational normal curve.  In the case of the rational normal curve, this is a classical result; see \cite[cf.~Proposition~6.1]{eGeomOSyz} for a proof.  In line with Philosophy \ref{phil:multigradedSyz}, we will see how the classical result gets deformed:  in the case of a (classical) rational normal curve of degree $d$, the minors come from a $2\times d$ linear matrix; in the type $d,e$ weighted case, we have a $2\times (d+e-1)$ matrix, all of whose entries are linear, with a single exception.  To be precise, we prove the following theorem:

\begin{theorem}
\label{thm:determinantal}
Let $I_2$ be the ideal of $S$ generated by the $2 \times 2$ minors of the following $2 \times (d+e-1$) matrix:
\begin{equation} \label{eq:definingmatrix}
M = \begin{bmatrix}
x_0 & x_1 & \ldots & x_{d-2} & x_{d-1}^e & y_0 & \ldots & y_{e-2}\\
x_1 & x_2 & \ldots & x_{d-1} & y_0 & y_1 & \ldots & y_{e-1}
\end{bmatrix}.
\end{equation}

Recall that $I$ is the defining ideal of the weighted rational curve of type $d,e$.
Then $I = I_2$.
\end{theorem}

Note that for any column of $M$, the ratio of the entries after applying $\varphi$ (recall that $\varphi$ was defined in Section \ref{sec:Background}) is $\frac{t}{s}$.  Thus, each minor of $M$ maps to $0$ under $\varphi$, so we know that $I_2 \subseteq I$.  To prove that these ideals are in fact equal, we will prove that the Hilbert series of $S/I_2$ is equal to that of $S/I$.  This relies on the following two results.

\begin{proposition}\label{prop:hilbertseriesR}
The Hilbert function and Hilbert series of $R$ are
$$
h_R(a) = da+1-a \bmod e \ \text{ and } \
HS_{R}(z) = \frac{1+(d-1)(z+z^2+\ldots+z^e)+(e-1)z^e}{(1-z)(1-z^e)}.
$$
\end{proposition}

\begin{proposition}\label{prop:hilbseriesI2}
The Hilbert series of $S/I_2$ is:
$$
HS_{S/I_2}(z) = \frac{1}{(1-z)^d(1-z^e)^e}\left( 1+ \sum_{k=1}^{d+e-2} (-1)^k k z^{k+1}\sum_{i=0}^{k+1}{{d-1}\choose{k+1-i}}{e\choose i}z^{i(e-1)}\right).
$$
\end{proposition}

Before proving these two propositions, we will show how they combine to yield a proof of Theorem~\ref{thm:determinantal}, the main result of this section.  The difficulty lies in proving that the two Hilbert series above are, in fact, equal.  In particular, as we observed in Example \ref{ex:d6e2}, the Hilbert series for $S/I_2$ has redundancy which is obscured in that for $R$.

\begin{proof}[Proof of Theorem \ref{thm:determinantal}]
Since $I_2 \subseteq I$, there is a surjection $S/I_2 \twoheadrightarrow R$.  If this map is an isomorphism in each degree, then it is an isomorphism of rings.  To show this, we will argue that $S/I$ and $S/I_2$ have the same Hilbert series.

In the Hilbert series for $S/I_2$, we can rewrite the double sum by substituting in $k+1-i+i-1$ for $k$:
$$
\sum_{k=1}^{d+e-2}\sum_{i=0}^{k+1}(-1)^k((k+1-i)+i-1)\binom{{d-1}}{{k+1-i}}z^{k+1-i}\binom{e}{i}(z^e)^i.
$$
Distributing over $(k+1-i)+i-1$ splits this sum into the following three pieces, each of which we will handle separately:
\begin{align*}
HS_{S/I_2}(z) = \frac{1}{(1-z)^d(1-z^e)^e}\Bigg( 1&+  \sum_{k=1}^{d+e-2}\sum_{i=0}^{k+1}(-1)^k(k+1-i)\binom{{d-1}}{{k+1-i}}z^{k+1-i}\binom{e}{i}(z^e)^i \\
&+ \sum_{k=1}^{d+e-2}\sum_{i=0}^{k+1}(-1)^ki\binom{d-1 }{ k+1-i}z^{k+1-i}\binom{e}{i}(z^e)^i \\
&+  \sum_{k=1}^{d+e-2}\sum_{i=0}^{k+1}(-1)^{k+1}\binom{d-1}{ k+1-i}z^{k+1-i}\binom{e}{i}(z^e)^i\Bigg).
\end{align*}
We will refer to the three sums as $p_1(z)$, $p_2(z)$, and $p_3(z)$, respectively.
The first sum, $p_1(z)$, can be handled as follows: 
\begin{align*}
p_1(z)&=\sum_{k=1}^{d+e-2}\sum_{i=0}^{k+1}(-1)^k(k+1-i)\binom{{d-1}}{{k+1-i}}z^{k+1-i}\binom{e}{i}(z^e)^i.\\
\intertext{By the identity $(k+1-i) \binom{d-1}{k+1-i} = (d-1)\binom{d-2}{k-i}$ and dropping the non-contributing $i=k+1$ term, this becomes:}
&=(d-1)z\sum_{k=1}^{d+e-2}\sum_{i=0}^{k}(-1)^k\binom{d-2}{k-i} z^{k-i}\binom{e}{i}(z^e)^{i}.\\
\intertext{By splitting $(-1)^k$ across $z^{k-i}$ and $(z^e)^{i}$, we obtain:}
&= (d-1)z\sum_{k=1}^{d+e-2}\sum_{i=0}^{k}(-z)^{k-i}\binom{d-2}{k-i}(-z^e)^{i} \binom{e}{i}.\\
\intertext{Finally, observe that, if the outer sum started at $k=0$, the double sum would be $(1-z)^{d-2}(1-z^e)^e$ on the nose.  So the sum becomes:}
&= (d-1)z\left[(1-z)^{d-2}(1-z^e)^e- 1\right].
\end{align*}
By a similar argument (now using that $i \binom{e}{i} = e \binom{e-1}{i-1})$), we can rewrite $p_2(z)$:
\begin{align*}
p_2(z)&=\sum_{k=1}^{d+e-2}\sum_{i=0}^{k+1}(-1)^ki\binom{d-1}{k+1-i}z^{k+1-i}\binom{e}{i}(z^e)^i\\
&= ez^e\sum_{k=1}^{d+e-2}\sum_{i=0}^{k}(-z)^{k-i}\binom{d-1}{k-i}(-z^e)^i\binom{e-1}{i}\\
&= ez^e\left[(1-z)^{d-1}(1-z^e)^{e-1}-1\right].
\end{align*}
Similarly, for $p_3(z)$ we have:
\begin{align*}
p_3(z)&=\sum_{k=1}^{d+e-2}\sum_{i=0}^{k+1}(-1)^{k+1}\binom{d-1}{k+1-i}z^{k+1-i}\binom{e}{i}(z^e)^i\\
&= \sum_{k=2}^{d+e-1}\sum_{i=0}^{k}(-z)^{k-i}\binom{d-1}{k-i}(-z^e)^i\binom{e}{i}\\
&= (1-z)^{d-1}(1-z^e)^e - (1-(d-1)z-ez^e).
\end{align*}
Putting this together and canceling appropriately, we have
\begin{align*}
HS_{S/I_2}(z) &= \frac{1}{(1-z)^d(1-z^e)^e}\left( 1+p_1(z)+p_2(z)+p_3(z)\right)\\
&=\frac{(d-1)z(1-z)^{d-2}(1-z^e)^e+ez^e(1-z)^{d-1}(1-z^e)^{e-1}+ (1-z)^{d-1}(1-z^e)^e}{(1-z)^d(1-z^e)^e}\\
&= \frac{1+(d-1)(z+z^2+\ldots+z^e)+(e-1)z^e}{(1-z)(1-z^e)}.
\end{align*}
This is equal to $HS_{S/I}(z)$, completing the proof.
\end{proof}

We delay the proofs of Propositions \ref{prop:hilbertseriesR} and \ref{prop:hilbseriesI2} to Subsection \ref{subsec:proofs}, and first give some corollaries.

\subsection{Corollaries}
\label{subsec:corollaries}

Theorem \ref{thm:determinantal}, together with the proof of Proposition \ref{prop:hilbseriesI2}, gives us the resolution of $R$ over $S$, and its graded Betti numbers.

\begin{corollary}\label{cor:eagonnorthcott}
$R$ is resolved by the Eagon--Northcott complex for $M$, and its graded Betti numbers are 
\[
\beta_{k,k+1+i(e-1)}(R) = k \left( \binom{d-1}{k+1-i} \binom{e}{i}\right) \text{ for } i = 0, \ldots, k+1
\]
and $0$ otherwise, for all $k=1, \ldots, d+e-2$.. 
\end{corollary}

Theorem \ref{thm:determinantal} also leads to an alternative proof that $R$ is Cohen-Macaulay.

\begin{corollary}
\label{cor:alternateCohenMacaulay}
$R$ is Cohen--Macaulay.
\end{corollary}

\begin{proof}
Because $R$ is a semigroup ring (see the proof of Proposition~\ref{prop:hilbertseriesR} for this perspective) where the semigroup has dimension $2$, $R$ also has dimension $2$. This means that $I$ has codimension $d+e-2$. By Corollary \ref{cor:eagonnorthcott}, the projective dimension of $R$ is also $d+e-2$.  Therefore, $R$ is Cohen-Macaulay.
\end{proof}

\begin{example}
As before, consider the weighted rational curve where $d=3$ and $e=2$.  Recall (from Proposition \ref{prop:hilbseriesI2}) that the Hilbert series for $R$ is given by:
\[
HS_R(z) = \frac{1-z^2-4z^3+3z^4+4z^5-3z^6}{(1-z)^3(1-z^2)^2}.
\]
The Betti table for $R$ over $S$ is given below:
\[\begin{matrix}
 & 0 & 1 & 2 & 3\\
%\text{total:} & 1 & 6 & 8 & 3\\
0: & 1 & . & . & .\\
1: & . & 1 & . & .\\
2: & . & 4 & 4 & .\\
3: & . & 1 & 4 & 3
\end{matrix}\]
Note that the Hilbert series does not entirely ``predict'' the Betti table.  That is, there is cancellation: both $\beta_{1,4}(R)$ and $\beta_{2,4}(R)$ are nonzero, so both of these contribute -- with opposite sign -- to the $3z^4$ term in the numerator of $HS_R(z)$. The reduced Hilbert series for $R$ (from Proposition \ref{prop:hilbertseriesR}) is:
\[
HS_R(z) = \frac{1+2z+3z^3}{(1-z)(1-z^2)}.
\]
\end{example}

\subsection{Proofs of Propositions \ref{prop:hilbertseriesR} and \ref{prop:hilbseriesI2}}
\label{subsec:proofs}

\begin{proof}[Proof of Proposition \ref{prop:hilbertseriesR}]

We directly compute the size of the monomial basis for a given degree. It will be convenient to view $R$ as the semigroup algebra
\[
\bbk[\Lambda] = \bbk[s^d, s^{d-1}t, \ldots, st^{d-1}, s^{e-1}t^{de-e+1}, s^{e-2}t^{de-e+2}, \ldots, st^{de-1}, t^{de}].
\]
Note that, by the usual normalization from the degree-$d$ Veronese map, monomials in $\bbk[\Lambda]$ of total degree $da$ correspond to monomials in $R$ of degree $a$.

By virtue of the monomials in the log complete series, every monomial $s^u t^{k\cdot de - u}$ is in $R$ as in the usual rational normal curve, so $\dim_\bbk (R)_{qe} = dqe+1$. To compute $\dim_\bbk R_{qe+r}$ where $1 \leq r \leq e-1$, we can count how many missing monomials there are from the usual Veronese basis, which has size $d(qe+r)+1$. It suffices to do this for $q=0$, since monomials of larger degree can be written $(s^u t^{kde-u}) \cdot s^i t^{\ell d-i}$ where $\ell < e$, using the guaranteed monomials of total power $k \cdot de$. Now observe that there are $r$ monomials missing in $R_r$: $\{t^{rd}, st^{rd-1}, s^2t^{rd-2}, \ldots, s^{r-1}t^{rd-r+1}\}$. Therefore $\dim_\bbk R_a = \dim_\bbk \bbk[\Lambda]_{da} = da + 1 - a \bmod e$. One can find a more general version of such counts in \cite[Proposition~6.3]{beLinSyzOCurvesIWProjSp}, where $Q$ denotes the ``missing'' monomials (though $d$ and $e$ must be suitably normalized). 

Now we manipulate the given generating function to get an expression for the Hilbert series: 
\[
\begin{aligned}
HS_{R}(z) &= \sum\limits_{a=0}^\infty (da+1 - a \bmod e) z^a \\
&= d\sum\limits_{a=0}^\infty az^a + \sum\limits_{a=0}^\infty z^a - \sum\limits_{a=0}^\infty (a \bmod e) z^a \\
&= dz \sum\limits_{a=0}^\infty \binom{a+1}{1}z^a + \sum\limits_{a=0}^\infty z^a - (z + 2z^2 + \cdots + (e-1)z^{e-1}) \sum\limits_{i=0}^\infty (z^e)^i \\ 
&= \frac{dz}{(1-z)^2} + \frac{1}{1-z} - \frac{z + 2z^2 + \cdots + (e-1)z^{e-1}}{1-z^e} \\ 
&= \frac{1 + (d-1)(z+z^2 + \cdots + z^e)+(e-1)z^e}{(1-z)(1-z^e)}
\end{aligned}
\]
\end{proof}

\begin{figure}
\centering
\begin{subfigure}{.5\textwidth}
  \centering
\begin{tikzpicture}[scale=0.5]
\foreach \x in {0,...,9}
\draw[line width=.6pt,color=black!15,dashed] (\x,0-.2) -- (\x,18+.2);
\foreach \y in {0,...,18}
\draw[line width=.6pt,color=black!15,dashed] (0-.2,\y) -- (9+.2,\y);
\draw[->] (0-.2,0) -- (9+.4,0) node[right] {$i$};
\draw[->] (0,0-.2) -- (0,18+.4) node[above] {$j$};
%\foreach \x in {0,...,-1}
%\draw[shift={(\x,0)},color=black] (0pt,2pt) -- (0pt,-2pt) node[below] {\footnotesize $\x$};
\foreach \x in {0,3,6,9}
\draw[shift={(\x,0)},color=black] (0pt,2pt) -- (0pt,-2pt) node[below] {\footnotesize $\x$};
%\foreach \y in {0,...,-1}
%\draw[shift={(0,\y)},color=black] (2pt,0pt) -- (-2pt,0pt) node[left] {\footnotesize $\y$};
\foreach \y in {0,3,6,9,12,15,18}
\draw[shift={(0,\y)},color=black] (2pt,0pt) -- (-2pt,0pt) node[left] {\footnotesize $\y$};
% \coordsys{0}{0}{9}{9}
% \foreach \s in {0,...,3}{
%     \foreach \t in {0,...,\s} %%won't let me do arithmetic in here to have t depend on, say, 2s
%         \draw [fill=black] (3*\s-\t,\t) circle (2pt);}
% \foreach \s/\t in {0/0, 3/0, 6/0, 9/0, 18/0,
%                     2/1, 5/1, 8/1, 11/1,
%                     1/2, 4/2, 7/2, 10/2}
%     \draw [fill=black] (\s,\t) circle (2pt);
\foreach \s in {0,...,3} \foreach \t in {0,...,6} \draw [fill=black] (3*\s, 3*\t) circle (4pt);

\foreach \s in {1,...,3} \foreach \t in {0,...,5} \draw [fill=black] (3*\s-1, 3*\t+1) circle (4pt);

\foreach \s in {1,...,3} \foreach \t in {0,...,5} \draw [fill=black] (3*\s-2, 3*\t+2) circle (4pt);

\foreach \s/\t in {0/3, 0/9, 0/15} \draw [color=black, fill=white] (\s, \t) circle (4pt);
\end{tikzpicture}
\caption{$d=3,e=2$}
\end{subfigure}%
\begin{subfigure}{.5\textwidth}
  \centering
\begin{tikzpicture}[scale=0.5]
\foreach \x in {0,...,9}
\draw[line width=.6pt,color=black!15,dashed] (\x,0-.2) -- (\x,18+.2);
\foreach \y in {0,...,18}
\draw[line width=.6pt,color=black!15,dashed] (0-.2,\y) -- (9+.2,\y);
\draw[->] (0-.2,0) -- (9+.4,0) node[right] {$i$};
\draw[->] (0,0-.2) -- (0,18+.4) node[above] {$j$};
%\foreach \x in {0,...,-1}
%\draw[shift={(\x,0)},color=black] (0pt,2pt) -- (0pt,-2pt) node[below] {\footnotesize $\x$};
\foreach \x in {0,3,6,9}
\draw[shift={(\x,0)},color=black] (0pt,2pt) -- (0pt,-2pt) node[below] {\footnotesize $\x$};
%\foreach \y in {0,...,-1}
%\draw[shift={(0,\y)},color=black] (2pt,0pt) -- (-2pt,0pt) node[left] {\footnotesize $\y$};
\foreach \y in {0,3,6,9,12,15,18}
\draw[shift={(0,\y)},color=black] (2pt,0pt) -- (-2pt,0pt) node[left] {\footnotesize $\y$};
% \coordsys{0}{0}{9}{9}
% \foreach \s in {0,...,3}{
%     \foreach \t in {0,...,\s} %%won't let me do arithmetic in here to have t depend on, say, 2s
%         \draw [fill=black] (3*\s-\t,\t) circle (2pt);}
% \foreach \s/\t in {0/0, 3/0, 6/0, 9/0, 18/0,
%                     2/1, 5/1, 8/1, 11/1,
%                     1/2, 4/2, 7/2, 10/2}
%     \draw [fill=black] (\s,\t) circle (2pt);
\foreach \s in {0,...,3} \foreach \t in {0,...,6} \draw [fill=black] (3*\s, 3*\t) circle (4pt);
\foreach \s in {1,...,3} \foreach \t in {0,...,5} \draw [fill=black] (3*\s-1, 3*\t+1) circle (4pt);
\foreach \s in {1,...,3} \foreach \t in {0,...,5} \draw [fill=black] (3*\s-2, 3*\t+2) circle (4pt);
\foreach \s/\t in {0/3, 0/6, 0/12, 0/15, 1/5, 1/14} \draw [color=black, fill=white] (\s, \t) circle (4pt);
\end{tikzpicture}
\caption{$d=3,e=3$}
\end{subfigure}%
\caption{Monomials $s^it^j$ appearing in $\bbk[\Lambda]$ for two choices of $d$ and $e$.}
\end{figure}

\begin{proof}[Proof of Proposition \ref{prop:hilbseriesI2}]
First, note that $I_2 \cong I_v/(v-x_{d-1}^e)$ where $\deg(v)=e$ and $I_v \subseteq S[v]$ is the ideal of $2 \times 2$ minors of:
$$
M_v = \begin{bmatrix}
x_0 & x_1 & \ldots & x_{d-2} & v & y_0 & \ldots & y_{e-2}\\
x_1 & x_2 & \ldots & x_{d-1} & y_0 & y_1 & \ldots & y_{e-1}
\end{bmatrix}.
$$
Since $M_v$ is a $(d+e-1) \times 2$ one-generic matrix, $I_v$ has codimension $d+e-2$ and $S' = S[v]/I_v$ is a Cohen--Macaulay integral domain by \cite[Theorem 6.4]{eGeomOSyz}.  Since $S/I_2 \cong S'/(v-x_{d-1}^e)$, the codimension of $I_2$ is also $d+e-2$.  Then by \cite[Theorem A2.60]{eGeomOSyz}, $S/I_2$ is resolved by the Eagon-Northcott complex of $M$, which allows us to compute the Hilbert series of $S/I_2$. It is worth noting that $I_v$ defines the rational normal scroll $\calS(d-1,e)$ in $\bbP^{d+e}$ (cf. \cite{ghPrinciplesAlgGeom}, \cite{ehVarietiesMinDeg}).

The Eagon--Northcott complex \cite{enIdDefBMatAACComplexAssocWThem} can be used to resolve determinantal ideals of the correct codimension. We follow the exposition in \cite[Appendix~A2H]{eGeomOSyz} and specialize to our needs. Use $EN_k$ to denote the $k$th module in the Eagon--Northcott resolution of $I_2$, defined by the maximal minors of the $2 \times (d+e-1)$-matrix $M$ from \eqref{eq:definingmatrix}. By definition, 
\[
EN_k \coloneqq (\Sym_{k-1} G) \otimes \textstyle{\bigwedge}^{k+1} F
\]
for $k \geq 1$, where $\Sym G$ is the symmetric algebra on the vector space $G$, whose basis elements correspond to the two rows of $M$, and $\textstyle{\bigwedge} F$ is the exterior algebra on the vector space $F$, whose basis elements correspond to the $d+e-1$ columns of $F$. Note then, that $F$ has $d-1$ degree-$1$ basis elements and $e$ degree-$e$ basis elements. 

The rank of $\Sym_{k-1}G$ is $k$, since $M$ has two rows. We count the various basis elements of $\textstyle{\bigwedge}^{k+1}F$ and stratify by degrees: there are $\binom{d-1}{k+1-i} \cdot \binom{e}{i}$ elements of degree $k+1-i + ie$ where $i$ ranges from $0$ to $k+1$. In particular, this means the Betti numbers of $S/I_2$ are
\[
\beta_{k,k+1+i(e-1)}(S/I_2) = k \left( \binom{d-1}{k+1-i} \binom{e}{i}\right) \text{ for } i = 0, \ldots, k+1
\]
and $0$ otherwise, for all $k=1, \ldots, d+e-2$. 

Assembling these Betti numbers into the Hilbert series and including $\beta_{0,0}=1$, we get the desired result up to some mild reorganization: 
\[
HS_{S/I_2}(z) = \frac{1}{(1-z)^d(1-z^e)^e} \left(1 + \sum\limits_{k=1}^{d+e-2} \sum\limits_{i=0}^{k+1} (-1)^k k \binom{d-1}{k+1-i} \binom{e}{i} z^{k+1+i(e-1)}\right).
\]

\end{proof}

\section{Gr\"obner bases}\label{sec:GB}

The matrix minors which define the usual rational normal curve are in fact a Gr\"obner basis for the defining ideal with respect to a graded reverse lexicographic order (\cite{sGrobnerBasesConvexPolytopes, eInitialIdealsVeroneseSubrings}, \cite[Theorem~3.1.1]{cdrKoszulAlgRegularity}).  We will show that the same is true for the minors which define the weighted rational curves.

Let $\bfw$ be the weight vector of all $1$s with the exception of $0$ in the $d$th position, i.e. the vector assigning weight $0$ to $x_{d-1}$ and weight $1$ to all other variables. Take $<_\bfw$ to be the term order which first takes leading forms with respect to $\bfw$, then breaks ties with revlex where $y_0 > \cdots > y_{e-1} > x_0 > \cdots > x_{d-1}$. 

\begin{theorem}
\label{thm:GBforI}
Recall that $I$ is the defining ideal of the weighted rational curve of type $d,e$ and that $M$ is the matrix whose $2 \times 2$ minors generate $I$. For the term order $<_\bfw$ above, the minors of $M$ form a Gr\"obner basis for $I$. 
\end{theorem}

\begin{remark}
The following proof does not rely on Theorem \ref{thm:determinantal}, and in fact offers an alternative proof that $I = I_2$.  We include both proofs since they offer distinct methods which may prove useful in generalizations.
\end{remark}

\begin{proof}

Let $G = \set{g_1, \ldots, g_\ell}$ be the set of $2 \times 2$ minors of $M$, and $J$ be the monomial ideal $J = \ideal{\initial_{<_\bfw} g_1, \ldots, \initial_{<_\bfw} g_\ell}$. Note that $J \subseteq \initial_{<_\bfw}(I_2)$ and $I_2 \subseteq I$.  These inclusions, together with the fact that Hilbert series are unchanged when we pass to an initial ideal (\cite[Theorem~15.26]{eCommAlgWAViewTAlgGeom}), give us the following string of inequalities:
\[
h_{S/J}(a) \geq h_{S/\initial_{<_\bfw}(I_2)}(a) = h_{S/I_2}(a) \geq h_R(a).
\]
We want to prove that $G$ is a Gr\"obner basis for $I$.  In light of the above remark, we will in fact prove instead that $G$ is a Gr\"obner basis for $I_2$ and that $I_2 = I$.  It suffices to show that $HS_{S/J}(z) = HS_{S/\initial_{<_\bfw}(I_2)}(z)$ and that $HS_{S/{I_2}}(z) = HS_R(z)$.  By the above inequalities, we need only to show that $S/J$ and $R$ have the same Hilbert series.

The following table describes the generators of $J$: 
\[
\begin{array}{c|l|l}
g & \text{index range} & \initial_{<_\bfw}(g) \\
\hline 
    % \begin{vmatrix}
    %     x_i & x_j \\
    %     x_{i+1} & x_{j+1}
    % \end{vmatrix} = 
    x_i x_{j+1} - x_j x_{i+1} & 
    0 \leq i < j \leq d-2 & 
    x_{i+1}x_j \\ 
\hline 
    % \begin{vmatrix}
    %     y_i & y_j \\ 
    %     y_{i+1} & y_{j+1}
    % \end{vmatrix} = 
    y_i y_{j+1} - y_j y_{i+1} & 
    0 \leq i < j \leq e-2 &
    y_{i+1}y_j \\
\hline 
    % \begin{vmatrix}
    %     x_j & y_i \\ 
    %     x_{j+1} & y_{i+1}
    % \end{vmatrix} = 
    x_j y_{i+1} - y_ix_{j+1} & 
    0 \leq i \leq e-2, \ 0 \leq j \leq d-2 &
    x_j y_{i+1} \\
\hline 
    % \begin{vmatrix}
    %     x_i & x_{d-1}^e \\ 
    %     x_{i+1} & y_0 
    % \end{vmatrix} = 
    x_i y_0 - x_{d-1}^e x_{i+1} &
    0 \leq i \leq d-2 &
    x_i y_0 \\
\hline 
    % \begin{vmatrix}
    %     x_{d-1}^e & y_i \\
    %     y_0 & y_{i+1}
    % \end{vmatrix} = 
    x_{d-1}^e x_{i+1} - y_0y_i &
    0 \leq i \leq e-2 & 
    y_0y_i
\end{array}
\]

We can more concisely display the information of what quadratics appear as generators of $J$ in the following table (note that the table is symmetric, so the entries below the diagonal are redundant information): 
\[
\begin{array}{c|c|c|c|c|c||c|c|c|c|c}
& x_0 & x_1 & \cdots & x_{d-2} & x_{d-1} & y_0 & y_1 & \cdots & y_{e-2} & y_{e-1} \\ 
\hline 
x_0 & & & & & & \checkmark & \checkmark & \checkmark & \checkmark & \checkmark \\
\hline 
x_1 & & \checkmark & \checkmark & \checkmark & & \checkmark & \checkmark & \checkmark & \checkmark & \checkmark \\ 
\hline 
\vdots & & \checkmark & \checkmark & \checkmark & & \checkmark & \checkmark & \checkmark & \checkmark & \checkmark \\
\hline 
x_{d-2} & & \checkmark & \checkmark & \checkmark & & \checkmark & \checkmark & \checkmark & \checkmark & \checkmark \\
\hline 
x_{d-1} & & & & & & & & & & \\
\hline 
\hline 
y_0 & \checkmark & \checkmark & \checkmark & \checkmark & & \checkmark & \checkmark & \checkmark & \checkmark & \\
\hline 
y_1 & \checkmark & \checkmark & \checkmark & \checkmark & & \checkmark & \checkmark & \checkmark & \checkmark & \\ 
\hline 
\vdots & \checkmark & \checkmark & \checkmark & \checkmark & & \checkmark & \checkmark & \checkmark & \checkmark & \\
\hline 
y_{e-2} & \checkmark & \checkmark & \checkmark & \checkmark & & \checkmark & \checkmark & \checkmark & \checkmark & \\
\hline 
y_{e-1} & \checkmark & \checkmark & \checkmark & \checkmark & & & & & \\
\end{array}
\]

\newpage
In other words, the only surviving quadratic monomials in $S/J$ are of the form: 
\begin{itemize}
    \item $x_0x_i$ where $0 \leq i \leq d-1$
    \item $x_ix_{d-1}$ where $0 \leq i \leq d-1$ 
    \item $x_{d-1}y_i$ where $0 \leq i \leq e-1$
    \item $y_i y_{e-1}$ where $0 \leq i \leq e-1$
\end{itemize}

With the above information, we can make the following observations in the interest of counting basis elements of $S/J$ of arbitrary degree. 
\begin{itemize}
    \item Monomials that are purely products of $x$s are of the form $x_0^a x_i x_{d-1}^b$ where $0 \leq i \leq d-1$. 
    \item Monomials that contain both an $x$ and a $y$ are of the form $x_{d-1}^a y_i y_{e-1}^b$ where $0 \leq i \leq e-1$.
    \item Monomials that are purely products of $y$s are of the form $y_iy_{e-1}^a$ where $0 \leq i \leq e-1$. 
\end{itemize}

Finally, we count basis elements toward assembling the Hilbert function of $S/J$. To obtain a monomial of degree $qe+r$, where $0 \leq q$ and $0 \leq r < e$, we need a product of $p$ $y$'s and $(q-p)e+r$ $x$'s, where $0 \leq p \leq q$. If $p=0$, the monomials can be of the form $x_0^ax_ix_{d-1}^b$ where $a+1+b = qe+r$ and $0 \leq i \leq d-1$. There are $(d-2)(qe+r)$ such monomials if $i \neq 0, d-1$ and $qe+r+1$ such monomials if $i=0$ or $d-1$, for a total of $(d-1)(qe+r)+1$ ``purely $x$'' monomials of degree $qe+r$. If $1 \leq p \leq q$, the monomial must have the form $x_{d-1}^{(q-p)e+r} y_i y_{e-1}^{p-1}$ where $0 \leq i \leq e-1$. There are $e$ such monomials for every $p$, leading to $qe$ monomials of degree $qe+r$ divisible by a $y$. Ultimately, we have $(d-1)(qe+r)+1 + qe = d(qe+r) + 1 - r$ monomials of degree $qe+r$. 

This shows that $S/J$ has the same Hilbert function as $R$, which proves both that $I=I_2$ and that $G$ is a Gr\"obner basis for $I$.
\end{proof}

Note that this Gr\"obner basis is not quadratic -- at least not in the usual sense -- so it does not imply that $R$ is Koszul as in the case of the usual rational normal curve.

\section{Koszulness}\label{sec:Koszulness}

The coordinate ring of the usual rational normal curve is Koszul (\cite{sGrobnerBasesConvexPolytopes, eInitialIdealsVeroneseSubrings}, \cite[Theorem~3.1.1]{cdrKoszulAlgRegularity}).  That is, the minimal graded free resolution of $\bbk$ over this ring is linear. In the case of a weighted rational curve, the resolution of $\bbk$ over $R$ cannot possibly be linear, at least in the traditional sense.  For instance, the first steps of the resolution will look like:
\[
0 \gets \bbk \gets R \gets [R(-1)^d \oplus R(-e)^e] \gets \cdots
\]
Nonetheless, it turns out that $R$ satisfies a generalized notion of Koszulness, introduced in \cite{hrwTKoszulPropIAffSemigroupRings}.

\begin{theorem}
\label{thm:Koszul}
The ring $(R,\frakm)$ is nonstandard Koszul.
\end{theorem}

Before defining ``nonstandard Koszul,'' we motivate the definition and result by examining the resolution of $\bbk$ over $R$.  Although not linear, this resolution appears well-behaved upon computational inspection and has a linear part which resolves $\bbk$ over a related ring. 

\begin{example} 

Take $d=3$ and $e=2$. The resolution of $\bbk$ over $R = S/I$ is of the form $R \leftarrow R^5 \leftarrow R^{16} \leftarrow R^{48} \leftarrow \cdots$ (suppressing twists) where the map $R^{16} \rightarrow R^5$ is given by the following matrix:
$$
\setcounter{MaxMatrixCols}{20}
\begin{bmatrix}
-x_1 & -x_2 & -y_0 & -y_1 & 0 & 0 & 0 & 0 & 0&0&0&0 & -y_0 & 0 & 0 & 0\\
x_0 & x_1 & \mathbf{x_2^2} & y_0 & -x_1 & -x_2 & -y_0 & -y_1 & 0&0&0&0& 0 & -y_0 & 0 & 0 \\
0&0&0&0 & x_0 & x_1 & \mathbf{x_2^2} & y_0 & 0&0&0&0 & 0&0& -y_0 & -y_1\\
0&0&0&0&0&0&0&0  & -x_1 & -x_2 & -y_0 & -y_1 & x_0 & x_1 & x_2 &0\\
0&0&0&0&0&0&0&0 & x_0 & x_1 & \mathbf{x_2^2} & y_0 & 0 & 0 & 0 & x_2
\end{bmatrix}.
$$ 

Naively taking only the linear entries of the above maps gives a resolution of $\bbk$ over $R_0 = S/I_0$, where $I_0 = \ideal{x_0x_2-x_1^2, x_0y_0, x_0y_1 - x_1y_0, x_1y_0, x_1y_1 - x_2y_0, y_0^2}$, which is somehow a forced homogenization of $I$ under the standard grading.  
In other words, the resolution of $\bbk$ over $R_0$ also has the form $R_0 \leftarrow R_0^5 \leftarrow R_0^{16} \leftarrow R_0^{48} \leftarrow \cdots$, and the map $R_0^{16} \rightarrow R_0^5$ is given by
$$
\setcounter{MaxMatrixCols}{20}
\begin{bmatrix}
-x_1 & -x_2 & -y_0 & -y_1 & 0 & 0 & 0 & 0 & 0&0&0&0 & -y_0 & 0 & 0 & 0\\
x_0 & x_1 & \mathbf{0} & y_0 & -x_1 & -x_2 & -y_0 & -y_1 & 0&0&0&0& 0 & -y_0 & 0 & 0 \\
0&0&0&0 & x_0 & x_1 & \mathbf{0} & y_0 & 0&0&0&0 & 0&0& -y_0 & -y_1\\
0&0&0&0&0&0&0&0  & -x_1 & -x_2 & -y_0 & -y_1 & x_0 & x_1 & x_2 &0\\
0&0&0&0&0&0&0&0 & x_0 & x_1 & \mathbf{0} & y_0 & 0 & 0 & 0 & x_2
\end{bmatrix}.
$$
This matrix is almost identical to the one above: the only changed entries are designated in bold, which are the ones previously containing $x_2^2$.
However, this new map is linear with respect to the standard grading.
\end{example}

Additionally, we have seen that it is possible to take an initial ideal of $I$ which is ``quadratic'' in the sense that it is generated by twofold products of variables, though these are not degree-$2$ terms in our grading. Taken together, these two observations suggest another instance of Philosophy \ref{phil:multigradedSyz} here.  Namely, the coordinate rings of weighted rational curves satisfy a Koszul property, but to understand this analogue, we need to deform the notion of Koszulness (in the manner of \cite[Definition~5.1]{hrwTKoszulPropIAffSemigroupRings}).

\begin{definition}
Let $(A, \mathfrak{n})$ be a (graded) local ring. The \textit{associated graded algebra} $\gr_\frak{n}(A)$ of $A$ with respect to the ideal $\frak n$ is the graded ring
\[
A/\frak n \oplus \frak n/\frak n^2 \oplus \frak n^2/\frak n^3 \oplus\cdots.
\] 
For any element $f$ of $A$, the \textit{initial form of} $f$, written $\initial_{\frak n}(f)$, is the image of $f$ in $\frak n^j/\frak n^{j+1}$ where $j$ is the greatest number such that $f \in \frak n^j$.  For an ideal $J$ of $A$, $\initial_{\frak n}(J)$ is the ideal of $\gr_\frak{n}(A)$ generated by $\initial_{\frak n}(f)$ for $f \in J$.
\end{definition}

\begin{definition}(\cite[cf.~Definition~5.1]{hrwTKoszulPropIAffSemigroupRings})
A ring $A$ which is either Noetherian local or $\bbN^d$-graded is \textit{nonstandard Koszul} if its associated graded ring $\gr_{\frak n} A$ is Koszul in the usual sense.
\end{definition} 

Note that if $A$ is standard graded, then $A$ is nonstandard Koszul if and only if $A$ is Koszul.  This is due to the fact that $\gr_{\frak n} A \cong A$.

As a first example where $A$ is nonstandard graded, suppose $A$ is any $\mathbb Z$-graded polynomial ring.  In this case, $\gr_{\frak n}(A)$ is a standard graded polynomial ring, so $A$ is nonstandard Koszul.

\begin{example}
    Let $A = \bbk[x,y]/\langle y^2 -xy^2 -x^4 \rangle$ where $\deg(x) = 1$ and $\deg(y) = 2$.  Then $\gr_{\frak n}(A) \cong \bbk[x,y]/\langle y^2 \rangle$ where now $\deg(x) = \deg(y) = 1$.  Because $\gr_{\frak n}(A)$ is Koszul, $A$ is nonstandard Koszul.
\end{example}

\begin{theorem}
The $\bbN^2$-graded Hilbert series of $\grmR$ induced by the natural $\bbN$-grading on $\grmR$ and the weighted $\bbN$-grading on $R$ where $\deg(x_i)=1$ and $\deg(y_i)=e$ is 
\[
HS_\grmR(z,w) = \frac{1+(d-2)wz + (e-1)w^ez - (d+e-2)w^{e-1}z^2}{(1-mz)^2(1-m^ez)}.
\]
Furthermore, the usual $\bbN$-graded Hilbert series of $\grmR$ is
\[
HS_\grmR(z) = \frac{1+(d+e-2)z}{(1-z)^2}. 
\]
\end{theorem}
\begin{proof}
The basis elements for $\grmR_u$ are monomials where $u$ is the maximal number of variables $x$ and $y$ appearing in a factorization of that monomial. Note that, to maximize the number of variables used, one must maximize the number of $x$'s appearing, since $x$'s have degree $1$. 

Monomials using $u$ variables must use $i$ $y$'s and $u-i$ $x$'s. For $i=0$, i.e. the monomials that can be ``maximally'' written entirely as a product of $u$ $x$'s, there are $(d-1)u+1$ such monomials. A monomial that is a product of $u$ $x$'s has weighted degree $u$, so this contributes $((d-1)u+1)w^u$ to the coefficient of $z^u$ in $HS_\grmR(z,w)$. For $1 \leq i \leq u$, there are $e$ monomials using $i$ $y$'s and $u-i$ $x$'s, each of weighted degree $ie+(u-i)$, contributing $e(w^{u+i(e-1)})$ to the coefficient of $z^u$ in $HS_\grmR(z,w)$.
Therefore, the coefficient of $z^u$ in $HS_\grmR(z,w)$ is 
\[
[z^u] HS_\grmR(z,w) = ((d-1)u+1)w^u + \sum\limits_{i=1}^u e w^{u+i(e-1)}.
\]

The bigraded Hilbert series of $\grmR$ therefore can be computed: 
\[
\begin{aligned}
HS_\grmR(z,w) &= \sum\limits_{u=0}^\infty ((d-1)u+1)w^uz^u + \sum\limits_{u=0}^\infty e \left( \sum\limits_{i=1}^u w^{(e-1)i}\right)(wz)^u \\
&= (d-1)\sum\limits_{u=0}^\infty u(wz)^u + \frac{1}{1-wz} + e \sum\limits_{u=0}^\infty \frac{m^{e-1}(1-w^{(e-1)u}}{1-mw^{e-1}} (wz)^u \\ 
&= \frac{(d-1)wz}{(1-wz)^2} + \frac{1-wz}{(1-wz)^2} + \frac{ew^{e-1}}{1-w^{e-1}} \left(\frac{1}{1-wz} - \frac{1}{1-w^ez}\right) \\ 
&= \frac{1+(d-2)wz}{(1-wz)^2} + \frac{ew^{e-1}}{1-w^{e-1}} \cdot \frac{wz(1-w^{e-1})}{(1-wz)(1-w^ez)} \\ 
&= \frac{(1+(d-2)wz)(1-w^ez)+ew^{e}z(1-wz)}{(1-wz)^2(1-w^ez)}\\
&=\frac{1+(d-2)wz+(e-1)w^ez-(d+e-2)w^{e+1}z^2}{(1-wz)^2(1-w^ez)}.
\end{aligned}
\]

To obtain the usual $\mathbb N$-graded Hilbert series of $\grmR$, we can evaluate $HS_\grmR(z,w)$ at $w=1$:
\[
\begin{aligned}
HS_\grmR(z) &= \left[HS_\grmR(z,w)\right]_{w \mapsto 1}\\
&= \frac{1+(d+e-2)z}{(1-z)^2}.
\end{aligned}
\]

Note, too, that $[HS_\grmR(z,w)]_{z \mapsto 1}$ yields $HS_R(w)$ from Proposition~\ref{prop:hilbertseriesR}, as expected.
\end{proof}

\begin{figure}
\centering
\begin{subfigure}{.5\textwidth}
  \centering
\begin{tikzpicture}[scale=0.5]

%Axes, labels, and gridlines:

\foreach \x in {0,...,9}
\draw[line width=.6pt,color=black!15,dashed] (\x,0-.2) -- (\x,18+.2);
\foreach \y in {0,...,18}
\draw[line width=.6pt,color=black!15,dashed] (0-.2,\y) -- (9+.2,\y);
\draw[->] (0-.2,0) -- (9+.4,0) node[right] {$i$};
\draw[->] (0,0-.2) -- (0,18+.4) node[above] {$j$};

\foreach \x in {0,3,6,9}
\draw[shift={(\x,0)},color=black] (0pt,-6pt) -- (0pt,-10pt) node[below] {\footnotesize $\x$};
\foreach \y in {0,3,6,9,12,15,18}
\draw[shift={(0,\y)},color=black] (-6pt,0pt) -- (-10pt,0pt) node[left] {\footnotesize $\y$};

%Dots for monomials:

\foreach \s in {0,...,3} \foreach \t in {0,...,6} \draw [fill=black] (3*\s, 3*\t) circle (4pt);

\foreach \s in {1,...,3} \foreach \t in {0,...,5} \draw [fill=black] (3*\s-1, 3*\t+1) circle (4pt);

\foreach \s in {1,...,3} \foreach \t in {0,...,5} \draw [fill=black] (3*\s-2, 3*\t+2) circle (4pt);

\foreach \s/\t in {0/3, 0/9, 0/15} \draw [color=black, fill=white] (\s, \t) circle (4pt);

%\foreach \s in {0,...,3} \foreach \t in {0,...,6} \draw [fill=black] (3*\s, 3*\t) circle (4pt);
%\foreach \s in {1,...,3} \foreach \t in {0,...,5} \draw [fill=black] (3*\s-1, 3*\t+1) circle (4pt);
%\foreach \s in {1,...,3} \foreach \t in {0,...,5} \draw [fill=black] (3*\s-2, 3*\t+2) circle (4pt);
%\foreach \s/\t in {0/3, 0/6, 0/12, 0/15, 1/5, 1/14} \draw [color=black, fill=white] (\s, \t) circle (4pt);

%%%%%%%% degree 1 noodle below

\draw [thick] (3.5,0) arc (415:225:0.35);
\draw [thick] (3.05,-0.55) -- (0.5,2); 
\draw [thick] (1.5,2) -- (3.5,0);

%\draw [thick] (0.5,2)--(0.5,5); 
%\draw [thick] (1.5,2) -- (1.5,5);

\path[draw, -] (0.5,2) [bend right=15] to (0.5,5);
\path[draw, -] (1.5,2) [bend left=15] to (1.5,5);

\draw [thick] (0.5,5)--(-0.5,6); 
\draw [thick] (1.5,5) -- (0,6.5);
\draw [thick] (0,6.5) arc (45:225:0.35);

%%%%%%%%%%%%%%%%%%%%%%%%% degree 2 noodle below

\draw [thick] (6.5,0) arc (415:225:0.35);
\draw [thick] (6.5,0) -- (2.5,4);
\draw [thick] (6.05,-0.55) -- (1.5,4);

\path[draw, -] (1.5,4) [bend right=15] to (1.5,7);
\path[draw, -] (2.5,4) [bend left=15] to (2.5,7);

\draw [thick] (2.5,7) -- (1.5,8);
\draw [thick] (1.5,7) -- (0.5,8);

\path[draw, -] (0.5,8) [bend right=15] to (0.5,11);
\path[draw, -] (1.5,8) [bend left=15] to (1.5,11);

\draw [thick] (0,12.5) arc (45:225:0.35);
\draw [thick] (1.5,11) -- (0, 12.5);
\draw [thick] (0.5,11) -- (-0.5,12);

%%%%%%%%%%%%%%%%%%%%%%%%%% degree 3 noodle below 

\draw [thick] (9.5,0) arc (415:225:0.35);
\draw [thick] (9.5,0) -- (3.5,6);
\draw [thick] (9.05,-0.55) -- (2.5,6);

\path [draw, -] (2.5,6) [bend right=15] to (2.5,9);
\path [draw, -] (3.5,6) [bend left=15] to (3.5,9);

\draw [thick] (3.5,9) -- (2.5,10);
\draw [thick] (2.5,9) -- (1.5,10);

\path [draw, -] (2.5,10) [bend left=15] to (2.5,13);
\path [draw, -] (1.5,10) [bend right=15] to (1.5,13);

\draw [thick] (2.5,13) -- (1.5,14);
\draw [thick] (1.5,13) -- (0.5,14);

\path [draw, -] (1.5,14) [bend left=15] to (1.5,17);
\path [draw, -] (0.5,14) [bend right=15] to (0.5,17);

\draw [thick] (0,18.5) arc (45:225:0.35);
\draw [thick] (1.5,17) -- (0, 18.5);
\draw [thick] (0.5,17) -- (-0.5,18);

\end{tikzpicture}
\caption{$d=3,e=2$}
\end{subfigure}%
\begin{subfigure}{.5\textwidth}
  \centering
\begin{tikzpicture}[scale=0.5]
\foreach \x in {0,...,9}
\draw[line width=.6pt,color=black!15,dashed] (\x,0-.2) -- (\x,18+.2);
\foreach \y in {0,...,18}
\draw[line width=.6pt,color=black!15,dashed] (0-.2,\y) -- (9+.2,\y);
\draw[->] (0-.2,0) -- (9+.4,0) node[right] {$i$};
\draw[->] (0,0-.2) -- (0,18+.4) node[above] {$j$};

\foreach \x in {0,3,6,9}
\draw[shift={(\x,0)},color=black] (0pt,-6pt) -- (0pt,-10pt) node[below] {\footnotesize $\x$};
\foreach \y in {0,3,6,9,12,15,18}
\draw[shift={(0,\y)},color=black] (-6pt,0pt) -- (-10pt,0pt) node[left] {\footnotesize $\y$};

%\foreach \x in {0,...,3}
%\draw[shift={(3*\x,0)},color=black] (0pt,2pt) -- (0pt,-2pt) node[below] {\footnotesize $\x$};
%\foreach \y in {0,...,6}
%\draw[shift={(0,3*\y)},color=black] (2pt,0pt) -- (-2pt,0pt) node[left] {\footnotesize $\y$};

\foreach \s in {0,...,3} \foreach \t in {0,...,6} \draw [fill=black] (3*\s, 3*\t) circle (4pt);
\foreach \s in {1,...,3} \foreach \t in {0,...,5} \draw [fill=black] (3*\s-1, 3*\t+1) circle (4pt);
\foreach \s in {1,...,3} \foreach \t in {0,...,5} \draw [fill=black] (3*\s-2, 3*\t+2) circle (4pt);
\foreach \s/\t in {0/3, 0/6, 0/12, 0/15, 1/5, 1/14} \draw [color=black, fill=white] (\s, \t) circle (4pt);

%%%%%%%% degree 0 noodle below 

%%%%%%%% degree 1 noodle below

\draw [thick] (3.5,0) arc (415:225:0.35);
\draw [thick] (3.05,-0.55) -- (0.5,2); 
\draw [thick] (1.5,2) -- (3.5,0);

\path[draw, -] (0.5,2) [bend right=15] to (1.5,7);
\path[draw, -] (1.5,2) [bend left=15] to (2.5,7);

\draw [thick] (0,9.5) arc (45:225:0.35);
\draw [thick] (2.5,7) -- (0,9.5);
\draw [thick] (1.5,7) -- (-0.5,9);

%%%%%%%%%%%%%%%%%%%%%%%%% degree 2 noodle below

\draw [thick] (6.5,0) arc (415:225:0.35);
\draw [thick] (6.5,0) -- (2.5,4);
\draw [thick] (6.05,-0.55) -- (1.5,4);

\path[draw, -] (1.5,4) [bend right=15] to (2.5,9);
\path[draw, -] (2.5,4) [bend left=15] to (3.5,9);

\draw [thick] (2.5,9) -- (0.5,11);
\draw [thick] (3.5,9) -- (1.5,11);

\path[draw, -] (0.5,11) [bend right=15] to (1.5,16);
\path[draw, -] (1.5,11) [bend left=15] to (2.5,16);

\draw [thick] (0,18.5) arc (45:225:0.35);
\draw [thick] (1.5,16) -- (-0.5, 18);
\draw [thick] (2.5,16) -- (0,18.5);

%%%%%%%%%%%%%%%%%%%%%%%%%% degree 3 noodle below 

\draw [thick] (9.5,0) arc (415:225:0.35);
\draw [thick] (9.5,0) -- (3.5,6);
\draw [thick] (9.05,-0.55) -- (2.5,6);

\path [draw, -] (2.5,6) [bend right=15] to (3.5,11);
\path [draw, -] (3.5,6) [bend left=15] to (4.5,11);

\draw [thick] (3.5,11) -- (1.5,13);
\draw [thick] (4.5,11) -- (2.5,13);

\path [draw, -] (1.5,13) [bend right=15] to (2.5,18);
\path [draw, -] (2.5,13) [bend left=15] to (3.5,18);

\draw [thick] (2.5,18) -- (2,18.5);
\draw [thick] (3.5,18) -- (3,18.5);

\end{tikzpicture}
\caption{$d=3,e=3$}
\end{subfigure}%
\caption{Monomial bases $s^it^j$ for the graded components $\grmR_1$, $\grmR_2$, and $\grmR_3$ for two choices of $d$ and $e$.}
\end{figure}

It turns out that, like $R$, the associated graded algebra $\grmR$ can be described as a quotient of $S$ by a determinantal ideal.  Let $I_0$ be the ideal generated by the $2\times 2$ minors of 
$$
M_0 = \begin{bmatrix}
x_0 & x_1 & \ldots & x_{d-2} & 0 & y_0 & \ldots & y_{e-2}\\
x_1 & x_2 & \ldots & x_{d-1} & y_0 & y_1 & \ldots & y_{e-1}
\end{bmatrix}.
$$

\begin{theorem}
The associated graded algebra $\grmR$ is isomorphic to $S/I_0$, where $S$ has the standard grading.
\end{theorem} 
\begin{proof}
We know that $\grm (S/I) = \grm S/\initial_{\frakm}(I)$.  We also know that $\grm S$ is $S$ with the standard grading.  Therefore, we need to show that $S/\initial_{\frak m}(I) = S/I_0$.

We will first argue that $I_0 \subseteq \initial_{\frak m}(I)$. It suffices to show that each generator of $I_0$ is the initial form of an element of $I$.  In fact, it turns out that every generator of $I_0$ is the initial form of a minor of $M$.  The initial forms of these minors are:
\begin{itemize}
    \item $\initial_{\frak m}(x_ix_{j+1}-x_jx_{i+1}) = x_ix_{j+1}-x_jx_{i+1}$,
    \item $\initial_{\frak m}(x_iy_{j+1}-y_jx_{i+1}) = x_iy_{j+1}-y_jx_{i+1}$,
    \item $\initial_{\frak m}(y_iy_{j+1}-y_jy_{i+1}) = y_iy_{j+1}-y_jy_{i+1}$,
    \item $\initial_{\frak m}(x_iy_0-x_{d-1}^ex_{i+1}) = x_iy_0$, and
    \item $\initial_{\frak m}(x_{d-1}^ey_{i+1}-y_0y_i) = -y_0y_i$.
\end{itemize}
Each generator of $I_0$ is of one of these forms, so $I_0\subseteq \initial_{\frak m}(I)$.  This means that we have a surjection $S/I_0 \twoheadrightarrow S/\initial_{\frak m}(I)$.  To see that this is in fact an isomorphism, we will argue that $S/I_0$ and $S/\initial_{\frak m}(I)$ have the same Hilbert series.

To compute $HS_{S/I_0}(z)$, consider $S' = S[v]/I_v$ where $S[v]$ is now standard graded, and $I_v$ is again the ideal generated by $2 \times 2$-minors of $M_v$. We will make use of the isomorphism $S/I_0 \cong S'/\ideal{v}$. Since $S'$ is an integral domain, we have the following short exact sequence:
$$
0 \longleftarrow S/I_0 \longleftarrow S' \longleftarrow S' \longleftarrow 0,
$$
where the rightmost map is multiplication by $v$. Since Hilbert series are additive over short exact sequences, we have $HS_{S/I_0}(z) = (1-z)HS_{S'}(z)$.

We know the Hilbert series of $S'$ since $S'$ is the coordinate ring of a rational normal scroll (\cite[Proposition 3.6]{msTowardFreeResOScrolls}):
$$
HS_{S'}(z) = \frac{1+(d+e-2)z}{(1-z)^3}.
$$
Then we have
$$
HS_{S/I_0}(z) = \frac{1+(d+e-2)z}{(1-z)^2},
$$
which is the Hilbert series of $\grmR$.
\end{proof}

We recall some further definitions from \cite[p. 56]{hrwTKoszulPropIAffSemigroupRings}.

\begin{definition}
Let $A$ be a Noetherian standard $\mathbb N$-graded $\bbk$-algebra.  The \textit{multiplicity} of $A$ is $e(A) \coloneqq P(1)$ where the Hilbert series of $A$ is written as $HS_{A}(z) = \frac{P(z)}{Q(z)}$ with $\gcd(P(z),Q(z)) = 1$.  For an $\mathbb N^d$-graded ring $(A,\mathfrak{n})$, the \textit{multiplicity} of $A$, $e(A)$, is defined to be $e(\gr_\mathfrak{n}(A))$. By a result from \cite{aLocalRingsHighEmbDim}, $e(A) \geq \dim_\bbk \mathfrak{n}/\mathfrak{n}^2-\dim A + 1$ if $A$ is Cohen--Macaulay.  If equality holds, then $A$ has \textit{minimal multiplicity}.
\end{definition}

\begin{proof}[Proof of Theorem \ref{thm:Koszul}]
We will apply Theorem 5.2 of \cite{hrwTKoszulPropIAffSemigroupRings}. We know $R$ is Cohen--Macaulay, so we need to show it has minimal multiplicity.  We have
$$
HS_\grmR (z) = \frac{1+(d+e-2)z}{(1-z)^2}.
$$
In the notation from above, $P(z) = 1+(d+e-2)z$, so $e(R) = d+e-1$. One can check that the rank of the Jacobian $J$ is $0$ at the origin and use $\dim_\bbk \frakm/\frakm^2 = (d+e) - \rank J|_0$ (\cite[Theorem 5.1]{hartshorne1977graduate}) to see that $\dim_\bbk \frakm/\frakm^2= d+e$. Therefore, we have
$$
\dim_{\bbk} \frakm/\frakm^2-\dim R + 1 = (d+e)-2+1 = d+e-1.
$$
So $R$ has minimal multiplicity, and therefore is nonstandard Koszul by \cite[Theorem~5.2]{hrwTKoszulPropIAffSemigroupRings}.
\end{proof}

We note that there is an alternative proof of Theorem \ref{thm:Koszul} that involves finding a quadratic Gr\"obner basis for $I_0$.

\begin{proposition}
    The minors of $M_0$ form a Gr\"obner basis for $I_0$ with respect to the term order $<_\bfw$ described above. 
\end{proposition}

\begin{proof}
    Note that whether a set of generators is a Gr\"obner basis does not depend on grading, so we can consider $I$ and $I_0$ as ideals of the same ring.  As in Theorem \ref{thm:GBforI}, let $G = \set{g_1, \ldots, g_\ell}$ be the set of $2 \times 2$ minors of $M$, and let $J = \ideal{\initial_{<_\bfw} g_1, \ldots, \initial_{<_\bfw} g_\ell}$.  Also let $H = \{h_1,\ldots,h_l\}$ be the set of $2\times 2$ minors of $M_0$, and write $J_0 = \ideal{\initial_{<_\bfw} h_1, \ldots, \initial_{<_\bfw} h_\ell}$.

    We know that $J_0 \subseteq \initial_{<_\bfw}(I_0)$.  To prove that these ideals are in fact equal, we will argue that they have the same Hilbert series.  First, note that $J_0 = J$, so these ideals have the same Hilbert series.  Since $G$ is a Gr\"obner basis for $I$, $J = \initial_{<_\bfw}(I)$, so $J_0$ and $I$ have the same Hilbert series.

    Using notation from the proof of Proposition \ref{prop:hilbseriesI2}, we have $S/I_0\cong S'/\langle v\rangle$, and $I_0$ has codimension $d+e-2$.  Then by \cite[Theorem A2.60]{eGeomOSyz}, $S/I_0$ is resolved by the Eagon-Northcott complex of $M_0$.  This means that $I$ and $I_0$ have the same Betti numbers, and so the same Hilbert series.  Therefore, $J_0$ and $\initial_{<_\bfw}(I_0)$ have the same Hilbert series.
\end{proof}

As a corollary of Theorem \ref{thm:Koszul}, we get the graded (and therefore ungraded) Poincar\'e series of $R$.

\begin{definition}
Let $M$ be an $\bbN^d$-graded $A$-module with Betti numbers $\beta_{i,\alpha}$. The $\bbN^d$-graded Poincar\'e series of $M$ is 
\[\Poin_M^A(z,w_1, \ldots, w_d) = \sum\limits_{i=0}^\infty \sum\limits_{\alpha \in \bbN^d} \beta_{i,\alpha} w_1^{\alpha_1} \cdots w_d^{\alpha_d} z^i,\]
and the ungraded Poincar\'e series of $M$ is 
\[
\Poin_M^A(z) = \sum\limits_{i=0}^\infty \beta_i z^i.
\]
\end{definition}

\begin{corollary}\label{cor:gradedpoincare}
The graded Poincar\'e series of $R$ is 
\[\Poin^R_\bbk(z,w) = \dfrac{(1+wz)^2(1+w^ez)}{1-(d-2)wz-(e-1)w^ez-(d+e-2)w^{e+1}z^2},\]
and therefore the ungraded Poincar\'e series of $R$ is 
\[\Poin^R_\bbk(z) = \frac{(1+z)^2}{1-(d+e-2)z}.\]
\end{corollary}

\begin{proof} 
By \cite[Corollary~5.6]{hrwTKoszulPropIAffSemigroupRings}, the graded Poincar\'e series of $R$ is equal to $\frac{1}{HS_{\grmR}(-z,w)}$ and the ungraded Poincar\'e series of $R$ is equal to $\frac{1}{HS_{\grmR}(-z)}$, manipulations which easily yield the formulas above. 

It is worth noting that by examining the coefficient of $z^i$ in the ungraded Poincar\'e series, we find that the ungraded Betti numbers are 
\[
\beta_i^R(\bbk) = \begin{cases}
1 &\text{ if } i=0 \\
d+e &\text{ if } i=1 \\
(d+e-1)^2(d+e-2)^{i-2} &\text{ if } i \geq 2
\end{cases}
\]
and therefore depend only on the quantity $d+e$, rather than on the values of $d$ and $e$ individually.  
\end{proof}

With the graded Poincar\'e series in hand, we can make the following observation about the Betti table of $\bbk$ as an $R$-module by examining the coefficients of $w^iz^i$ and $w^iz^{ei}$ in $\Poin_\bbk^R(z,w)$.

\begin{corollary}
\label{cor:bettitable}
The resolution of $\bbk$ as an $R$-module has the following entries as part of its Betti table: 
\[
\beta_{i,i}^R(\bbk) = \begin{cases}
1 & \text{ if } i=0\\
d & \text{ if } i=1\\
(d-2)^{i-2}(d-1)^2 & \text{ if } i \geq 2
\end{cases}
\: \text{ and } \: 
\beta_{i,ei}^R(\bbk) = \begin{cases}
1 & \text{ if } i=0\\
e(e-1)^{i-1} & \text{ if } i\geq 1
\end{cases}.
\]
In particular, the resolution of $\bbk$ over $R$ has both an infinite ``true linear'' strand if $d > 2$ and a ``maximally twisted'' strand for all $e>1$, and the Betti numbers in these extremal strands depend only on $d$ and $e$ respectively. 
\end{corollary}

\begin{figure}
\centering
\begin{subfigure}{.33\textwidth}
  \centering

\[\begin{matrix}
 & 0 & 1 & 2 & 3 & 4\\
%\text{total:} & 1 & 4 & 9 & 18 & 36\\
0: & \mathbf 1 & 2 & 1 & . & .\\
1: & . & \mathbf 2 & 6 & 6 & 2\\
2: & . & . & \mathbf 2 & 10 & 18\\
3: & . & . & . & \mathbf 2 & 14\\
4: & . & . & . & . & \mathbf 2
\end{matrix}\]

\caption{$d=2,e=2$}
\end{subfigure}%
\begin{subfigure}{.33\textwidth}
  \centering

\[
\begin{matrix}
 & 0 & 1 & 2 & 3 & 4\\
%\text{total:} & 1 & 5 & 16 & 48 & 144\\
0: & \mathbf 1 & 3 & 4 & 4 & 4\\
1: & . & \mathbf 2 & 10 & 24 & 40\\
2: & . & . & \mathbf 2 & 18 & 72\\
3: & . & . & . & \mathbf 2 & 26\\
4: & . & . & . & . & \mathbf 2
\end{matrix}
\]

\caption{$d=3,e=2$}
\end{subfigure}%
\begin{subfigure}{.33\textwidth}
  \centering

\[
\begin{matrix}
 & 0 & 1 & 2 & 3 & 4\\
%\text{total:} & 1 & 6 & 25 & 100 & 400\\
0: & \mathbf 1 & 3 & 4 & 4 & 4\\
1: & . & . & . & . & .\\
2: & . & \mathbf 3 & 15 & 36 & 60\\
3: & . & . & . & . & .\\
4: & . & . & \mathbf 6 & 48 & 180\\
%5: & . & . & . & . & .\\
%6: & . & . & . & 12 & 132\\
%7: & . & . & . & . & .\\
%8: & . & . & . & . & 24
\end{matrix}
\]

\caption{$d=3,e=3$}
\end{subfigure}%
\caption{Betti tables for $\bbk$ over $R$ for three choices of $d$ and $e$. Entries in the ``maximally twisted'' strands are in bold.}
\label{fig:kBettiNumbers}
\end{figure}

\begin{example}
  We compare the cases $d=2$,  $e=2$; $d=3$, $e=2$; and $d=3$, $e=3$.  The Betti tables for $\bbk$ over $R$ in each of these three cases are shown in Figure \ref{fig:kBettiNumbers}.  Note that when $d=e=2$, the resolution does not have a ``true linear'' strand.  However, in both of the other cases, it does.  This strand depends only on $d$ and so is the same for the two cases $d=3$, $e=2$ and $d=e=3$.  All three resolutions have a ``maximally twisted'' strand.  This strand  depends only on $e$, and so the last case differs from the first two.  
\end{example}

The resolution of $\bbk$ over $R$ can be obtained by leveraging the close relation between $R$ and the rational normal scroll $\calS(d-1,e)$. 

\begin{proposition}
If $\bfF$ is the minimal free resolution of $\bbk[v]$ over the coordinate ring $S'$ of the scroll $\calS(d-1,e)$, then $\bfF \otimes_{S'} R$ gives a resolution of $\bbk$ over $R$. Similarly, $\bfF \otimes_{S'} \grmR$ gives a resolution of $\bbk$ over $\grmR$.
\end{proposition}

\begin{proof}
As observed in the proof of Proposition~\ref{prop:hilbseriesI2}, $R$ can be viewed as the quotient $S[v]/(I_v + \ideal{v-x_{d-1}^e})$, where $I_v$ is the ideal of $2\times2$ minors of 
\[
M_v = \begin{bmatrix}
x_0 & x_1 & \ldots & x_{d-2} & v & y_0 & \ldots & y_{e-2}\\
x_1 & x_2 & \ldots & x_{d-1} & y_0 & y_1 & \ldots & y_{e-1}
\end{bmatrix},
\]
while $S' \coloneqq S[v]/I_v$ is simply the coordinate ring of the scroll $\calS(d-1,e)$. Note that with this notation, $R \cong S'/\ideal{v-x_{d-1}^e}$. We will also use $\bfx$ and $\bfy$ to denote the sets of variables $x_0, \ldots, x_{d-1}$ and $y_0, \ldots, y_{e-1}$ respectively. 

Let $\bfF$ be the minimal free resolution of $S'/\ideal{\bfx, \bfy}$ over $S'$, which can be obtained as a variant of the resolution of $S'/\ideal{\bfx, v, \bfy}$ over $S'$ presented in \cite[Section~4]{msTowardFreeResOScrolls}. More precisely, one can construct the resolutions of $\ideal{\bfx}$ and $\ideal{\bfy}$ using the same blocks shown in \cite{msTowardFreeResOScrolls}, then zip them together with a mapping cone for a resolution of $\ideal{\bfx, \bfy}$, where the chain maps follow the same pattern with $y_0$ playing the role of $v$. We will show that $\bfF \otimes_{S'} S'/\ideal{v-x_{d-1}^e}$ is also exact, and thus resolves $R/\ideal{\bfx, \bfy}$ as desired. 

To this end, we check $\Tor_i^{S'}(S'/\ideal{\bfx, \bfy}, S'/\ideal{v-x_{d-1}^e})$. Because $\Tor$ is balanced, this can be computed from $S'/\ideal{\bfx, \bfy} \otimes_{S'} \bfG$, where $\bfG: 0 \leftarrow S' \xleftarrow{\cdot v-x_{d-1}^e} S' \leftarrow 0$ is the resolution of $S'/\ideal{v-x_{d-1}^e}$ over $S'$. Because $S'/\ideal{\bfx, \bfy} \cong \bbk[v]$, $S'/\ideal{\bfx, \bfy} \otimes_{S'} \bfG$ is simply $0 \leftarrow \bbk[v] \xleftarrow{v} \bbk[v] \leftarrow 0$, which only has homology at the $0$th spot, since multiplication by $v$ is injective. Therefore, $\bfF \otimes_{S'} S'/\ideal{v-x_{d-1}^e}$ is exact except at the zeroth spot, where its homology is $\bbk$, as desired. 

The same argument, using $\grmR \cong S[v]/(I_v + \ideal{v})$, gives the second statement.
\end{proof}

An immediate consequence of the Poincar\'e series in Corollary~\ref{cor:gradedpoincare} is the fact that $R$ is Golod. 

\begin{corollary}
\label{cor:Golod}
The ring $R$ is Golod.
\end{corollary}

\begin{proof}
With the computations already done, it is straightforward to check the defining inequality for Golodness: 
\[
\Poin_\bbk^R(z) = 
\frac{(1+z)^{d+e}}{1-z^2 \Poin^S_I(z)}.
\]

The Betti numbers for the resolution of $I$ over $S$ are given by a coarsened (shift of) the Eagon--Northcott count from Corollary~\ref{cor:eagonnorthcott}: $\beta_i^S(I) = (i+1) \binom{d+e-1}{i+2}$ for $i=0,\ldots, d+e-3$. We can thus compute:
\[
\begin{array}{ll}
1-z^2 \Poin^S_I(z) &= 1 - z^2 \left( \sum\limits_{i=0}^{d+e-3} (i+1) \binom{d+e-1}{i+2} z^i \right) \\
&= 1 - \sum\limits_{i=2}^{d+e-1} (i-1) \binom{d+e-1}{i} z^i \\ 
&= 1 - \sum\limits_{i=2}^{d+e-1}i \binom{d+e-1}{i} z^i + \sum\limits_{i=2}^{d+e-1} \binom{d+e-1}{i} z^i \\ 
&= 1 - \sum\limits_{i=2}^{d+e-1} (d+e-1) \binom{d+e-2}{i-1} z^i + \left[(1+z)^{d+e-1} - 1 - (d+e-1)z\right] \\ 
&= -(d+e-1)z \sum\limits_{i=1}^{d+e-2} \binom{d+e-2}{i} z^i + (1+z)^{d+e-1} - (d+e-1)z \\ 
&= -(d+e-1)z[(1+z)^{d+e-2}-1] + (1+z)^{d+e-1} - (d+e-1)z \\
&= (1+z)^{d+e-2}[1-(d+e-2)z]
\end{array}
\]
The equality follows immediately.
\end{proof}

\bibliographystyle{alpha}
\bibliography{bibliography}

\end{document}